\documentclass[a4paper, 12pt]{article}
\usepackage{amsmath}
\usepackage{amsfonts}
\usepackage{amsthm}
\usepackage{graphicx}
\usepackage{hyperref}

\newtheorem{teo}{Theorem}
\newtheorem{lem}{Lemma}
\newtheorem{prop}{Proposition}

\newtheorem{defin}{Definition}
\newtheorem*{mainteo}{The Slab Theorem}
\newtheorem*{mainlem}{Dragging Lemma}

\newcommand{\dist}{\operatorname{dist}}

\newcommand{\cyl}{\operatorname{Cyl}}

\newcommand{\psl}{\operatorname{PSL}}
\newcommand{\di}{\operatorname{div}}
\newcommand{\id}{\operatorname{Id}}

\title{The Slab Theorem for Minimal Surfaces in $\mathbb{E}(-1,\tau)$}
\author{Vanderson Lima\thanks{The author was supported by CNPq-Brazil.}}
\date{}

\begin{document}

\maketitle

\begin{abstract}
\noindent Unlike $\mathbb{R}^{3}$, the homogeneous spaces $\mathbb{E}(-1,\tau)$ have a great variety of entire vertical minimal graphs. In this paper we explore conditions which guarantees that a minimal surface in $\mathbb{E}(-1,\tau)$ is such a graph. More specifically: we introduce the definition of a generalized slab in $\mathbb{E}(-1,\tau)$ and prove that a properly immersed minimal surface of finite topology inside such a slab region has multi-graph ends. Moreover, when the surface is embedded, the ends are graphs. When the surface is embedded and simply connected, it is an entire graph.
\end{abstract}

\providecommand{\abs}[1]{\lvert#1\rvert}

\linespread{1} 

\section{Introduction}

An important problem in the theory of minimal surfaces is to classify such surfaces by their topological type. The simplest topological examples are properly embedded surfaces that are simply connected or homeomorphic to the annulus $\mathbb{S}^{1}\times\mathbb{R}$. This problem was extensively studied when the ambient space is $\mathbb{R}^{3}$. One of the first unicity results in this setting is the classical theorem of Bernstein, which states that a minimal graph defined over the plane $\mathbb{R}^{2}$ ({\it an entire graph}) is a plane. Generalizations of this theorem and other results were obtained over the years until the following theorem was proved by Meeks and Rosenberg:
\begin{teo}[\cite{M.R}] A properly embedded simply connected minimal surface in $\mathbb{R}^{3}$ is either a plane or a helicoid.
\end{teo}

In the case of an annulus we have the following unicity theorem due to Collin:
\begin{teo}[\cite{C}]
A proper minimal embedding of an annulus $\mathbb{S}^{1}\times\mathbb{R}$, in $\mathbb{R}^{3}$ is a catenoid.
\end{teo} 

In general, it is known what is the geometric behavior of the ends of a properly embedded minimal surface with finite topology: a proper minimal embedding of $\mathbb{S}^{1}\times\mathbb{R}^{+}$, is asymptotic to an end of a plane, helicoid or catenoid, see \cite{B.B, M.R}. 

We remark that in the previous theorems we can replace {\it properly embedded}, by {\it complete and embedded}, since Colding and Minicozzi proved that a complete embedded minimal surface of finite topology in $\mathbb{R}^{3}$ is proper, \cite{C.M2}. In general this is not true, even in the case where the undelying ambient space is homogeneous, see \cite{Cos} and \cite{R.T} where examples of complete non-proper minimal embeddings of the disc are constructed in $\mathbb{H}^{3}$ and $\mathbb{H}^{2}\times\mathbb{R}$ respectively.

In this paper we are interested in a certain family of homogeneous $3$-manifolds. Consider the Homogeneous spaces $\mathbb{E}(-1,\tau)$. 
These manifolds are characterized by the following properties: there is a Riemannian submersion $$\Pi: \mathbb{E}(-1,\tau) \to \mathbb{H}^{2},$$ such that the fibers of $\Pi$ are the integral curves of a unit Killing vector field $\xi$, and the constant $\tau$ (called the bundle curvature) satisfies $$\nabla_{X} \xi = \tau X\wedge\xi,$$ for any vector field $X$.
For each $\tau$, $\mathbb{E}(-1,\tau)$ is diffeomorphic to $\mathbb{H}^{2}\times\mathbb{R}$, however it is not a Riemannian product for $\tau \neq 0$.  We can define the notion of {\it vertical graph} as the image of a smooth section $s: \Omega \subset \mathbb{H}^{2} \rightarrow \mathbb{E}(-1,\tau)$, for some domain $\Omega$. More generally, we define a {\it vertical multi-graph} as a surface transverse to the vertical Killing field $\xi$. 

A first important result is that Bernstein's theorem is not valid for vertical graphs in these ambient manifolds: Nelli and Rosenberg \cite{N.R1, N.R2} proved that given any continuous rectifiable curve $\Gamma \subset \partial_{\infty}\mathbb{H}^{2}\times\mathbb{R}$, $\Gamma$ a graph over $\partial_{\infty}\mathbb{H}^{2}$, there is an entire minimal graph asymptotic to $\Gamma$ at infinity. This result was extended to the spaces $\mathbb{E}(-1,\tau)$ by Folha and Pe\~ nafiel, \cite{F.P}. 

There are many other examples of properly embedded simply connected minimal surfaces. The so called {\it Scherk surfaces}: if $\Gamma$ is an ideal polygon of $\mathbb{H}^{2}$, there are necessary and sufficient conditions on $\Gamma$ which ensure the existence of a minimal graph over its interior, taking values plus and minus infinity on alternate sides of $\Gamma$, \cite{C.R, Me}, so there are complete minimal graphs that are not entire. Finally, we mention the existence of families of minimal embeddings of the plane that are not graphs (we describe them in section 2). 

Due to this great variety of examples it seems to be very difficult to obtain some kind of classification. However, this discussion gives rise to a related and important problem: to find conditions that require a properly embedded minimal plane in $\mathbb{E}(-1,\tau)$ to be an entire graph. Collin, Hauswirth, Rosenberg proved that in $\mathbb{H}^{2}\times\mathbb{R}$ a sufficient condition is the surface be inside a slab of height less than $\pi$. In fact, they proved this condition is sufficient to assure that a properly immersed minimal surface of finite topology has multi-graph ends:

\begin{teo}[\cite{C.H.R2}]\label{CHR}
Let $\Sigma$ be a properly immersed minimal surface in $\mathbb{H}^{2}\times\mathbb{R}$ of finite topology, which is inside a slab of height $\pi - \epsilon$, for some $\epsilon > 0$. Then each of its ends is a multi-graph. Moreover: 
\begin{itemize}
\item[a)] If $\Sigma$ is embedded, then each of its ends contains an annulus that is a graph over the complement of a disc in $\mathbb{H}^{2}$;
\item[b)] If $\Sigma$ is embedded and has only one end, then it is an entire graph.
\end{itemize}
\end{teo}

The celebrated Half-Space Theorem of Hoffman and Meeks states that a properly immersed minimal surface in $\mathbb{R}^{3}$ which is above a plane is a parallel plane, \cite{H.M}. More generally the work of Colding and Minicozzi gave restrictions to the geometry of minimal surfaces inside slabs of $\mathbb{R}^{3}$, \cite{C.M}. So this inspires the slab condition of this theorem. The number $\pi$ is intrinsically related to the existence of some minimal surfaces in $\mathbb{H}^{2}\times\mathbb{R}$, for example vertical catenoids. Actually, there is an example of a properly embedded simply connected minimal surface inside a slab of height $\pi$, which is not a graph, so the number $\pi$ is optimal for this theorem. Also, the authors constructed an example of a properly immersed simply connected minimal surface which is inside a slab of height less than $\pi$, but is not a graph. Thus the condition of being embedded is necessary for the last conclusion.

However, there are entire minimal graphs that are not inside any slab of height less than $\pi$. We may ask: What are the regions of $\mathbb{H}^{2}\times\mathbb{R}$ for which theorem 3 is true? Also, one can ask if a theorem 3 is valid in the spaces $\mathbb{E}(-1,\tau)$, $\tau \neq 0$. The key properties of a slab used in the proof of theorem 3 are the fact that for each point $p$ inside the slab we can place a compact vertical rotational catenoid containing $p$, with boundary outside the slab, and the fact that the height function of any surface inside the slab is bounded. Inspired by this we introduce the following definition:

\begin{defin}
Consider the half-space model or the cylinder model of $\mathbb{E}(-1,\tau)$ (see section 2 and remark \ref{rem}). We say that a region $\mathcal{R}$ of $\mathbb{E}(-1,\tau)$ is a generalized slab if the following conditions are satisfied: 
\begin{enumerate}
\item $\mathcal{R}$ is a domain bounded by two disjoint entire vertical graphs $S_1$ and $S_2$ with bounded height (in relation to $\mathbb{H}^{2}\times\{0\}$) and such that the tangent planes of $S_1$ and $S_2$ are bounded away from the vertical, i.e, $\langle N_{i},\xi\rangle \geq c > 0$, where $N_{i}$ is the unit normal of $S_{i}$ and $c$ is a constant;
\item There is a $C^{1}$ map $\Psi:\mathcal{R}\times\bigl(\mathbb{S}^{1}\times [-1,1]\bigr) \to \mathbb{E}(-1,\tau)$, such that, for each $p \in \mathcal{R}$ we have that $\mathcal{C}(p) := \Psi\bigl(p,\mathbb{S}^{1}\times [-1,1]\bigr)$ is a minimal annulus containing the point $p$ and its two boundary curves are disjoint from $\mathcal{R}$, one above $\mathcal{R}$ and one below $\mathcal{R}$.  Moreover, any two annuli of the family $\{\mathcal{C}(p); p \in \mathcal{R}\}$ are isometric to each other.
\end{enumerate}
\end{defin}

\noindent
{\bf Remark 1:} A priori the definition of the height function in $\mathbb{E}(-1,\tau)$ depends on the model one considers. However, as proposition \ref{isometry} shows, there is an isometry between the half-space and the cylinder models which relates the height functions, showing that it differs by a bounded function. Therefore property $1$ in the above definition does not depend of the choice of one of these two models.\label{rem}
\\

In this setting we prove the following:
\begin{mainteo}
Let $\Sigma$ be a properly immersed minimal surface of finite topology in $\mathbb{E}(-1,\tau)$, $\tau \in \mathbb{R}$, inside of a generalized slab $\mathcal{R}$. Then each of its ends is a multi-graph. Moreover:
\begin{itemize}
\item[a)] If $\Sigma$ is embedded, then each of its ends contains an annulus that is a graph over the complement of a disc in $\mathbb{H}^{2}$;
\item[b)] If $\Sigma$ is embedded and has only one end, then it is an entire graph.
\end{itemize}
\end{mainteo}

The paper is organized as follows. In section 2 we fix some notation and present some results about the geometry of the spaces $\mathbb{E}(-1,\tau)$ and its minimal surfaces. In section 3 we prove our main result. Finally, in section 4 we construct some examples of generalized slabs.
\\\\
\textbf{Acknowledgments.} This work is part of the author's Ph.D. thesis at Instituto Naciona de Matem\'atica Pura e Aplicada (IMPA). The
author would like to express his sincere gratitude to his advisor Harold Rosenberg for his patience, constant encouragement and guidance. He would like also to thank Marco A. M. Guaraco for making the pictures that appear in the paper, Abigail Folha and Carlos Pe\~nafiel for provide the preprint \cite{F.P}, and Beno\^it Daniel, Jos\'e Espinar, Laurent Hauswirth, Laurent Mazet, L\' ucio Rodriguez, Magdalena Rodriguez and William Meeks III for discussions and their interesting in this work. Finally he also thanks the referee by the suggestions and corrections.

\section{The spaces $\mathbb{E}(-1,\tau)$}

Good references for the geometry of $\mathbb{E}(-1,\tau)$ are the papers \cite{F.P, M, P1, Y}, and the results not proved here can be found in them. 

For each $\tau$, the space $\mathbb{E}(-1,\tau)$ is diffeomorphic to $\mathbb{H}^{2}\times\mathbb{R}$. Let $(x,y) \mapsto \zeta (x,y)$ be a conformal parametrization of $\mathbb{H}^{2}$ and let $\lambda$ be the conformal factor such that the metric of $\mathbb{H}^{2}$, in these coordinates, is $\lambda^{2}\bigl(dx^2 + dy^2\bigr)$. Let $t \in (-\infty,+\infty)$ be the coordinate on the $\mathbb{R}$ factor. The metric of $\mathbb{E}(-1,\tau)$ in these coordinates is $$ds^{2} = \lambda^{2}\bigl(dx^2 + dy^2\bigr) + \biggl(2\tau\Bigl(\frac{\lambda_{y}}{\lambda}dx - \frac{\lambda_{x}}{\lambda}dy\Bigr) + dt^2\biggr)^2.$$ 

Observe that $E(-1,0)$ is isometric to $\mathbb{H}^{2}\times\mathbb{R}$. When $\tau = -1/2$ we have the following description: denote by $U\mathbb{H}^{2}$ the unit tangent bundle of $\mathbb{H}^{2}$ endowed with the Sasaki metric. The universal cover of this space endowed with the covering metric is isometric to $E(-1,-1/2)$, which corresponds to the {\it Thurston Geometry} $\widetilde\psl_{2}(\mathbb{R})$. 

The fibration over $\mathbb{H}^{2}$ on the coordinates introduced is given by $\Pi: \mathbb{E}(-1,\tau) \rightarrow \mathbb{H}^{2}$, $\Pi(x,y,t) = (x,y).$ Moreover the vertical unit Killing vector field is $\xi = \partial_{t}$. We can construct an orthonormal frame $\{E_{1}, E_{2}, E_{3}\}$ in the following way: consider the orthonormal frame of $\mathbb{H}^{2}$ given by $\{e_{1} = \lambda^{-1}\partial_{x}, e_{2} = \lambda^{-1}\partial_{y}\}$. Let $E_{1}$, $E_{2}$ be the horizontal lifts of $e_{1}$, $e_{2}$ respectively, and $E_{3} = \partial_{t}$, then $$d\Pi(E_{i}) = e_{i} \ \textrm{and} \ \langle E_{i},E_{3}\rangle = 0, \ i =1,2,$$ and in coordinates $E_{1} = \lambda^{-1}\partial_{x} - 2\tau\lambda^{-2}\lambda_{y}\partial_{t}, \ E_{2} = \lambda^{-1}\partial_{y} + 2\tau\lambda^{-2}\lambda_{x}\partial_{t}.$\\

We will wok with two models:
\begin{description}
\item [The half-space model] - Take the half-plane model of the hyperbolic plane, $$\mathbb{H}^{2} = \{(x, y) \in \mathbb{R}^{2}; y > 0\} \ \textrm{and} \ \lambda = \displaystyle\frac{1}{y};$$
\item[The cylinder model] - Take the disc model of the hyperbolic plane, $$\mathbb{H}^{2} = \{(u, v) \in \mathbb{R}^{2}; u^{2} + v^{2} < 1\} \ \textrm{and} \ \lambda = \displaystyle\frac{2}{1 - u^{2} - v^{2}}.$$
\end{description}

An isometry from the half-plane model to the disc model of $\mathbb{H}^{2}$ is given by the map $$\phi(x,y) = \biggl(\frac{2x}{x^{2} + (y + 1)^2},\frac{x^{2} + y^{2} - 1}{x^{2} + (y + 1)^2}\biggr).$$

\begin{lem}[\cite{F.P}]\label{isometry}
The map $\phi$ gives rise to an isometry from the half-space model to the cylinder model of $\mathbb{E}(-1,\tau)$, of the form $\Phi(x,y,t) = \bigl(\phi(x,y),w(x,y,t)\bigr)$, where $$w(x,y,t) = t + 4\tau\arctan\biggl(\frac{x}{y + 1}\biggr).$$
\end{lem}

We now describe the isometries of $\mathbb{E}(-1,\tau)$. Each such map $F$ is fiber preserving and induces an isometry $f$ of $\mathbb{H}^{2}$ such that $\Pi\circ F = f\circ\Pi$. We can actually give a formula in terms of $f$. Identify $\mathbb{R}^{2}$ with the set of complex numbers $\mathbb{C}$, so we can write $(x,y) \in \mathbb{H}^{2}$ as $z = x + iy$.

\begin{prop}
The isometries of $\mathbb{E}(-1,\tau)$ are given by:
\begin{enumerate}
\item In the half-space model $$F(z,t) = \bigl(f(z),t - 2\tau\arg f'(z) + c\bigr)$$ or $$G(z,t) = \bigl(-\overline{f(z)},-t + 2\tau\arg f'(z) + c\bigr),$$ where $f$ is a positive isometry of the half-plane model of $\mathbb{H}^{2}$, $c \in \mathbb{R}$ and $\arg f': \mathbb{H}^2 \rightarrow \mathbb{R}$ is a smooth angle function, i.e, $e^{i\arg f'(z)} = \frac{f'(z)}{|f'(z)|}$, which has to be chosen.
\item In the cylinder model $$F(z,t) = \bigl(f(z),t - 2\tau\arg f'(z) + c\bigr)$$ or $$G(z,t) = \bigl(\overline{f(z)},-t + 2\tau\arg f'(z) + c\bigr),$$ where $f$ is a positive isometry of the disc model of $\mathbb{H}^{2}$, $c \in \mathbb{R}$ and $\arg f': \mathbb{H}^2 \rightarrow \mathbb{R}$ is a smooth angle function, i.e, $e^{i\arg f'(z)} = \frac{f'(z)}{|f'(z)|}$, which has to be chosen.
\end{enumerate}
\end{prop}

Observe that the coordinate $t$ is not necessarily preserved by these isometries. This can be explained by the fact $\mathbb{E}(-1,\tau)$ is not a product space if $\tau \neq 0$, thus the $t$ coordinate is not geometrically invariant.

In the next lemma we describe in coordinates the horizontal lift of a curve $\alpha \in \mathbb{H}^{2}$.

\begin{lem}[\cite{F.P}]\label{boundt}
Let $\alpha$ be a connected curve in $\mathbb{H}^{2}$, and denote by $\tilde\alpha$ its horizontal lift to $\mathbb{E}(-1,\tau)$. Suppose $\tilde\alpha$ is contained in a slice $\{t = t_{0}\}$ then:
\begin{enumerate}
\item If we consider the disc model, $\alpha$ is contained in a complete geodesic passing through the origin $(0,0)$;
\item If we consider the half-plane model, $\alpha$ is contained in $\{x = x_0\}$ for some $x_0 \in \mathbb{R}$.
\end{enumerate}
Moreover, on the Half-space model, the horizontal lift of a geodesic has bounded $t$-coordinate. 
\end{lem}

\begin{proof}
Let $\alpha: I \rightarrow \mathbb{H}^{2}$ be a curve defined on a open interval $I$, $\alpha(s) = \bigl(x(s),y(s)\bigr)$. A horizontal lift has the form $\tilde\alpha(s) = \bigl(x(s),y(s),t(s)\bigr)$, and its velocity vector is given by $$\tilde\alpha'(s) = \lambda x'(s)E_{1} + \lambda y'(s)E_{2} + \bigl[t'(s) + 2\tau\lambda^{-1}\bigl(x'(s)\lambda_{y} - y'(s)\lambda_{x}\bigr)\bigr]E_{3}.$$
Since $\langle \tilde\alpha'(s),E_{3}\rangle = 0$ we have 
\begin{equation}
t'(s) + 2\tau\lambda^{-1}\bigl(x'(s)\lambda_{y} - y'(s)\lambda_{x}\bigr) = 0.
\end{equation}
1) In the disc model the ODE (1) has solution $$t(s) = t_{0} + 2\tau\int_{s_0}^{s} \lambda(r)\bigl[x(r)y'(r) - x'(r)y(r)\bigr]dr.$$
Then $\tilde\alpha(s)$ is inside a slice $\{t = t_{0}\}$ if, and only if, $y(s) = \pm ax(s), a > 0.$
\\\\
2) In the half-plane model the solution of (1) is $$t(s) = t_{0} + 2\tau\int_{s_0}^{s} \lambda(r)x'(r)dr.$$
Then $\tilde\alpha(s)$ is inside a slice $\{t = t_{0}\}$ if, and only if, $x(s) = \textrm{constant}, \forall s \in I$.
\\

Consider the half-space model of $\mathbb{E}(-1,\tau)$. Besides the curves $x = \textrm{constant}$, the geodesics of $\mathbb{H}^2$ in the half-plane model are the curves $\gamma: (0,2\pi) \rightarrow \mathbb{H}^2$ defined by $\gamma(\theta) = (R\cos\theta + x_0,R\sin\theta).$ So the $t$-coordinate of the horizontal lift of $\gamma$ is given by $t(\theta) = t_0 - 2\tau\theta$, hence is bounded.

\end{proof}

\subsection{CMC Graphs}\label{cmcgraph}

A surface in $\mathbb{E}(-1,\tau)$ is called a {\it vertical graph} if is the image of a smooth section $\sigma: \Omega \subset \mathbb{H}^{2} \rightarrow \mathbb{E}(-1,\tau)$ of the Riemannian submersion, i.e $\Pi\circ \sigma = \id_{\Omega}$, for some domain $\Omega$. For such a map let $u(x,y)$ be the signed distance from $(x,y,0)$ to $\sigma(x,y)$ along the geodesic fiber through $(x,y,0)$. This defines a function $u \in C^{0}(\partial\Omega)\cap C^{\infty}(\Omega)$ such that $$\Sigma(u) := \bigl\{\bigl(x,y,u(x,y) \in \mathbb{E}(-1,\tau); (x,y) \in \Omega\bigr)\bigr\}$$ is the image of the section $\sigma$. Respectively a function $u \in C^{0}(\partial\Omega)\cap C^{\infty}(\Omega)$ defines a smooth section.

The unit normal vector of $\Sigma(u)$ is given by $N = \frac{1}{W}( aE_{1} + bE_{2} + E_{3})$, where $a = \displaystyle -\frac{u_{x}}{\lambda} - 2\tau\frac{\lambda_{y}}{\lambda^{2}}$, $b = \displaystyle -\frac{u_{y}}{\lambda} + 2\tau\frac{\lambda_{x}}{\lambda^{2}}$ and $W = \sqrt{1 + a^2 + b^2}$.

\begin{lem}
Let $u: \Omega \subset \mathbb{H}^{2} \rightarrow \mathbb{R}$ be a function whose graph $\Sigma(u)$ in $\mathbb{E}(-1,\tau)$ has constant mean curvature $H$. Then $u$ satisfies $$\di_{\mathbb{H}^{2}}\biggl(\frac{a}{\lambda W}\partial_{x} + \frac{b}{\lambda W}\partial_{y}\biggr) = 2H.$$
\end{lem}

\subsection{Catenoids}\label{cat}

Catenoids in $\mathbb{H}^{2}\times\mathbb{R}$ were first constructed by Nelli and Rosenberg on \cite{N.R1,N.R2}. One important property of these surfaces is they exist precisely when their height is less than $\pi$. In the spaces $\mathbb{E}(-1,\tau)$ this kind of surface was constructed by C. Pe\~ nafiel in \cite{P1}. We describe them below.

In the disc model of $\mathbb{H}^{2}$ euclidean rotations about the origin are isometries. Lifting these maps to isometries in $\mathbb{E}(-1,\tau)$, the image of a point $(x,y,t)$ is obtained rotating by $\theta_{0}$ the $(x,y)$ part around the $t-$axis then translating it along the $t-$axis by $-2\tau\theta_{0}$. Now compose with translation by $2\tau\theta_{0}$ along the $t-$axis and we obtain one-parameter group of rotations along this axis. Rotational surfaces are those invariant by this group.

Consider polar coordinates $(x,y) = \bigl(\tanh(\frac{\rho}{2})\cos \theta, \tanh(\frac{\rho}{2})\sin \theta\bigr)$. The function $$u(\rho) = \int_{\arcsin (d)}^{\rho} \frac{d\sqrt{1 + 4\tau^{2}\tanh^{2}(\frac{r}{2})}}{\sqrt{\sinh^{2}r - d^{2}}}dr,$$
gives rise to a minimal bi-graph:

\begin{prop}[\cite{P1}]
For each $d \geq 0$ there exists a complete minimal rotational surface $\mathcal{C}_{d}$ in $\mathbb{E}(-1,\tau)$. The surface $\mathcal{C}_{0}$ is the slice $\{t = 0\}$. For $d > 0$ the rotational surface $\mathcal{C}_{d}$ (called catenoid) is embedded and homeomorphic to an annulus. The asymptotic boundary of $\mathcal{C}_{d}$ is two horizontal circles in $\partial_{\infty}(\mathbb{D})\times\mathbb{R}$ and the vertical distance between them is the non-decreasing function $\mathcal{H}$ defined by $$\mathcal{H}(d) = \int_{\arcsin (d)}^{+\infty} \frac{d\sqrt{1 + 4\tau^{2}\tanh^{2}(\frac{r}{2})}}{\sqrt{\sinh^{2}r - d^{2}}}dr,$$ and which satisfies $$\lim_{d \to 0} \mathcal{H}(d) = 0, \ \ \ \  \textrm{and} \ \ \ \ \lim_{d \to \infty} \mathcal{H}(d) = \pi\sqrt{1 + 4\tau^{2}}.$$ 
\end{prop}

\subsection{Some simply connected minimal surfaces} 

Beside graphs the simplest examples of simply connected surfaces in $\mathbb{E}(-1,\tau)$ are the vertical cylinders over curves, i.e, the surfaces of the kind $\Pi^{-1}(\alpha)$, where $\alpha$ is a curve in $\mathbb{H}^{2}$. The mean curvature $H$ of these surfaces satisfies $$2H(q) = k_{g}(p),$$ where $p = \Pi(q)$ and $k_{g}$ is the geodesic curvature of $\alpha$. Thus if $\alpha$ is a complete geodesic, $\Pi^{-1}(\alpha)$ is a properly embedded minimal surface. These are the only minimal surfaces invariant under any vertical translation.

\begin{figure}[!ht]
\centering
\includegraphics[scale=0.8]{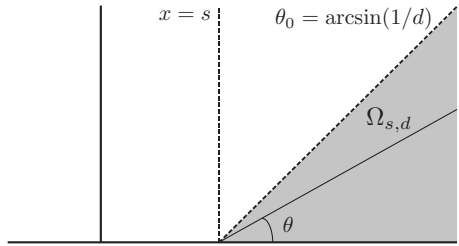}
\caption{The domain $\Omega_{s,d}$}
\label{fig5}
\end{figure}

We present now another family of simply connected properly embedded minimal surfaces obtained as invariant surfaces. Consider the half-plane model for the hyperbolic plane. Fix the point $(s,0) \in \partial_{\infty}\mathbb{H}^2$. We can introduce polar coordinates centered at $(s,0)$ associating for each $(\phi,\theta) \in \mathbb{R}\times (0,\pi)$ the point $(e^{\phi}\cos{\theta} + s,e^{\phi}\sin{\theta}) \in \mathbb{H}^2$. In these coordinates, the hyperbolic metric becomes $$\frac{1}{\sin^2{\theta}}(d\phi^2 + d\theta^2).$$
 
The geodesic $\gamma_{s} = \{(s,t); t > 0\}$ corresponds to the curve $\{\theta = \pi/2\}$ and each curve $\gamma_{s,\theta} = \{\theta = \theta_{0}\}$ is equidistant to this geodesic. The distance between the equidistant and the geodesic is given by $$\dist(\theta_0) = \biggl|\int_{\theta_0}^{\frac{\pi}{2}}\frac{d\theta}{\sin{\theta}}\biggr|.$$

\begin{prop}[\cite{F.P}]
Let $d > 1$. Define $u_{d}^{-}, u_{d}^{+}: \bigl(0,\arcsin (1/d)\bigr) \rightarrow \mathbb{R}$ by $$u_{d}^{+}(\theta) = \int_{\theta}^{\arcsin (1/d)}\frac{d\sqrt{1 + 4\tau^{2}\cos^{2}\psi}}{\sqrt{1 - d^{2}\sin^{2}\psi}}d\psi - 2\tau\bigl(\arcsin (1/d) + \theta\bigr),$$ and
$$u_{d}^{-}(\theta) = -\int_{\theta}^{\arcsin (1/d)}\frac{d\sqrt{1 + 4\tau^{2}\cos^{2}\psi}}{\sqrt{1 - d^{2}\sin^{2}\psi}}d\psi + 2\tau\bigl(\arcsin (1/d) - \theta\bigr),$$
and let $\Omega_{s,d}$ be the domain of $\mathbb{H}^{2}$ with boundary $\gamma_{s,\theta}\cup\{(x,0); x > s\}$. These functions define vertical minimal graphs over $\Omega_{s,d}$, invariant by hyperbolic translation over $\gamma_{s}$. Moreover, defining 
$$h(d) = \int_{0}^{\arcsin (\frac{1}{d})}\frac{d\sqrt{1 + 4\tau^{2}\cos^{2}\psi}}{\sqrt{1 - d^{2}\sin^{2}\psi}}d\psi,$$ we have
$$\lim_{d \to +\infty} h(d) = \frac{\pi}{2}\sqrt{1 + 4\tau^{2}}, \ \ \ \lim_{d \to 1^{+}} h(d) = +\infty.$$
The union of these graphs gives rise to a complete minimal surface $\mathcal{M}_{h}(s)$ whose projection on $\mathbb{H}^{2}$ is $\Omega_{s,d}$. The tangent planes of $\mathcal{M}_{h}(s)$ are vertical along points which project on the curve $\gamma_{s, \arcsin(1/d)}$. The asymptotic boundary is the union of $$\bigl\{(x,0,t); x \geq s, t = -h(d) + 2\tau\arcsin (1/d) \ \textrm{or} \ t = h(d) - 2\tau\arcsin (1/d)\bigr\}$$ and the vertical segments joining the end points of these arcs.
\end{prop}

\begin{proof}
The functions $u_{d}^{\pm}$ define minimal graphs invariant by hyperbolic translations, \cite{P1,Y}. We have $u_{d}^{-}\bigl(\arcsin (1/d)\bigr) = u_{d}^{+}\bigl(\arcsin (1/d)\bigr)$, and over the equidistant curve $\gamma_{s,\arcsin (1/d)}$ the graphs of $u_{d}^{-}$ and $u_{d}^{+}$ have vertical tangent planes. Thus the union of these graphs gives rise to a complete minimal surface $\mathcal{M}_{h}(s)$ whose projection on $\mathbb{H}^{2}$ is $\Omega_{s,d}$.

Now let us study the behavior of $h(d)$ when $d$ goes to $+\infty$ or to $1$. Making a change the variables $v = d\sin\psi - 1$ we obtain
$$h(d) = \int_{-1}^{0}\frac{\sqrt{d^{2}(1 + 4\tau^{2}) - 4\tau^{2}(v + 1)^{2}}}{\sqrt{[d^{2} - (v + 1)^{2}][1 - (v + 1)^{2}]}}dv,$$
thus $\displaystyle\lim_{d \to +\infty} h(d) = \frac{\pi}{2}\sqrt{1 + 4\tau^{2}}$.

Now make the change of variables $w = d\sin\psi$. When $d > 1$, we have
\begin{eqnarray*}
h(d) &=& \int_{0}^{\arcsin (\frac{1}{d})}\frac{d\sqrt{1 + 4\tau^{2}\cos^{2}\psi}}{\sqrt{1 - d^{2}\sin^{2}\psi}}d\psi\\
&\geq & \int_{0}^{\arcsin (\frac{1}{d})}\frac{1}{\sqrt{1 - d^{2}\sin^{2}\psi}}d\psi\\
&=& \int_{0}^{1}\frac{1}{\sqrt{d^{2} - w^{2}}\sqrt{1 - w^{2}}}dw.
\end{eqnarray*}
But,
$$\lim_{d \to 1^{+}}\int_{0}^{1}\frac{1}{\sqrt{d^{2} - w^{2}}\sqrt{1 - w^{2}}}dw = \int_{0}^{1}\frac{1}{1 - w^{2}}dw = +\infty,$$
so $\displaystyle\lim_{d \to 1^{+}} h(d) = +\infty$.
\end{proof}

\noindent
\textbf{Remark 2}: We decided to index the surfaces $\mathcal{M}_{h}(s)$ by the parameter $h$ instead of $d$, because $h$ determines the height of the surfaces. We can also define the functions of proposition $3$ on the interval $\bigl(\arcsin (1/d) + \pi/2,\pi\bigr)$, and this gives rises to similar minimal surfaces which we also denote by $\mathcal{M}_{h}(s)$. In this case, the projection in $\mathbb{H}^{2}$ is the domain with boundary $\gamma_{s,\theta}\cup\{(x,0); x < s\}$.\\

In the rest of this section, unless otherwise specified, we will work on the Half-space model.

Below we state some properties of this family of surfaces that will be useful to our work. 

\begin{lem}\label{transverse}
Let $S$ be an entire vertical graph in $\mathbb{E}(-1,\tau)$ (half-space model) such that $S \subset \mathbb{H}^{2}\times [-h_{0},h_{0}]$, for some $h_{0} > 0$, and $\langle N,\partial_{t}\rangle \geq c_{0} > 0$, for some $c_{0}$, where $N$ is the unit normal of $S$.

Then, for $h > h_{0}$ large the surface $\mathcal{M}_{h}(s)$ is transverse to $S$.
\end{lem}

\begin{proof}
We claim that as $h \to +\infty$ the unit normal of $\mathcal{M}_{h}(s)$ on the part inside $\mathbb{H}^{2}\times [-h_{0},h_{0}]$ tends to a horizontal vector. Observe that $h \to +\infty$ when $d \to 1^+$. We will prove the result for the graph of $u_{d}^{+}$, the case $u_{d}^{-}$ is analogous. The function $u_{d}^{+}$ only depends on $\theta$, so the unit normal is given by 
$$N_{d}^{+} = \frac{\Bigl(\frac{\partial u_{d}^{+}}{\partial\theta} \sin^{2}\theta + 2\tau\Bigr)E_{1} - \Bigl(\frac{\partial u_{d}^{+}}{\partial\theta} \sin\theta \cos\theta\Bigr)E_{2} + E_{3}}{\sqrt{1 + 4\tau^{2} + 4\tau\Bigl(\frac{\partial u_{d}^{+}}{\partial\theta}\Bigr)\sin^{2}\theta +\Bigl(\frac{\partial u_{d}^{+}}{\partial\theta}\Bigr)^{2}\sin^{2}\theta}},$$
hence 
$$\langle N_{d}^{+}, \partial_{t}\rangle = \frac{\sqrt{1 - d^{2}\sin^{2}\theta}}{\sqrt{1 + 4\tau^{2}\cos^{2}\theta}}.$$

Define $c_{\tau} = \min\{-2\tau\pi,2\tau\pi\}$. Given $\varepsilon > 0$ choose $\delta > 0$ sufficiently small such that, for $d \in (1, 1 + \delta)$ there is $\theta(d)$ with $u_{d}\bigr(\theta(d)\bigl) = h_{0}$ and $$\frac{4(2\delta + \delta^{2})e^{2(h_{0} - c_{\tau})}}{\bigl(2 + \delta(1 + e^{2(h_{0} - c_{\tau})})\bigr)^{2}} < \varepsilon^{2}.$$

Denote $\theta(\delta) = \arcsin (\frac{1}{1 + \delta})$. Since $1 < d < 1 + \delta,$ we have
\begin{eqnarray*}
h_{0} &=& \int_{\theta(d)}^{\arcsin (\frac{1}{d})}\frac{d\sqrt{1 + 4\tau^{2}\cos^{2}\psi}}{\sqrt{1 - d^{2}\sin^{2}\psi}}d\psi - 2\tau\bigl(\arcsin (1/d) + \theta\bigr) \\\\
&>& \int_{\theta(d)}^{\arcsin (\frac{1}{1 + \delta})}\frac{1}{\sqrt{1 - \sin^{2}\psi}}d\psi + c_{\tau}\\\\
&=& \log \Biggl(\frac{\bigl[\cos(\frac{\theta(\delta)}{2}) + \sin(\frac{\theta(\delta)}{2})\bigr]\bigl[\cos(\frac{\theta(d)}{2}) - \sin(\frac{\theta(d)}{2})\bigr]}{\bigl[\cos(\frac{\theta(\delta)}{2}) - \sin(\frac{\theta(\delta)}{2})\bigr]\bigl[\cos(\frac{\theta(d)}{2}) + \sin(\frac{\theta(d)}{2})\bigr]}\Biggr) + c_{\tau},
\end{eqnarray*}

After some algebraic manipulation we obtain
\begin{eqnarray*}
\sqrt{1 - \sin^{2}\bigl(\theta(d)\bigr)} < \frac{\sqrt{4(2\delta + \delta^{2})e^{2(h_{0} - c_{\tau})}}}{2 + \delta(1 + e^{2(h_{0} - c_{\tau})})} < \varepsilon.
\end{eqnarray*}

But $\theta > \theta(d)$ implies $u_{d}^{+}(\theta) < h_{0}$ and $$\frac{\sqrt{1 - d^{2}\sin^{2}\theta}}{\sqrt{1 + 4\tau^{2}\cos^{2}\theta}} < \sqrt{1 - d^{2}\sin^{2}\theta} < \sqrt{1 - \sin^{2}\bigl(\theta(d)\bigr)}.$$

Thus for $h$ large, $\mathcal{M}_{h}(s)$ is close to a vertical plane inside the slab. Moreover $S$ has tangent planes bounded away from the vertical. Therefore these surfaces are transverse. 
\end{proof}

Observe that the surface $\mathcal{M}_{h}(s)$ is parametrized by the two charts $\Psi^{\pm}(\phi,\theta) = \bigl(e^{\phi}\cos{\theta} + s,e^{\phi}\sin{\theta},u_{d}^{\pm}(\theta)\bigr)$. Consider the curves $\alpha^{\pm}_{\phi} = \Psi^{\pm}(\phi,\cdot)$ and $\beta^{\pm}_{\theta} = \Psi^{\pm}(\cdot,\theta)$. Both families $\{\alpha^{+}_{\phi}\cup\alpha^{-}_{\phi}\}$ and $\{\beta^{+}_{\theta}\cup\beta^{-}_{\theta}\}$ foliate $\mathcal{M}_{h}(s)$. The unit tangent vectors of $\alpha^{\pm}_{\phi}$ and $\beta^{\pm}_{\theta}$ are given respectively by $$V^{\pm} = \frac{-E_{1} + (\cot\theta)E_{2} + \bigl(2\tau + \frac{d u_{d}^{\pm}}{d\theta}\bigr)E_{3}}{\sqrt{1 + \cot^{2}\theta + \bigl(2\tau + \frac{d u_{d}^{\pm}}{d\theta}\bigr)^{2}}}, \ W^{\pm} = \frac{(\cot\theta)E_{1} + E_{2} - (2\tau\cot\theta)E_{3}}{\sqrt{1 + (1+ 4\tau^{2})\cot^{2}\theta}},$$
so, $$\langle V^{\pm},E_{3}\rangle = \frac{\mp d\sqrt{1 + 4\tau^{2}\cos^{2}\theta}}{\sqrt{(1 + \cot^{2}\theta)(1 -d^{2}\sin^{2}\theta) + d^2(1 + 4\tau^{2}\cos^{2}\theta)}},$$ $$\langle W^{\pm},E_{3}\rangle = \frac{-2\tau\cot\theta}{\sqrt{1 + (1+ 4\tau^{2})\cot^{2}\theta}}.$$

Thus, for $d$ close to $1$ and $\theta$ close to $\pi/2$ we have that $V^{\pm}$ is almost vertical and $W^{\pm}$ is almost horizontal. Hence, $\mathcal{F} := \{\alpha_{\phi} := \alpha^{+}_{\phi}\cup\alpha^{-}_{\phi}\}\cap\bigl(\mathbb{H}^{2}\times [-h_{0},h_{0}]\bigr)$ is a foliation of this slab by compact curves which are almost vertical. 

\begin{prop}\label{separation}
Let $h$, $h_0$ and $S$ as in lemma \ref{transverse}. Then $\mathcal{M}_{h}(s)$ separates $S$ in two connected components.
\end{prop}

\begin{proof}
Since the surfaces are transverse their intersection is a union of curves. Let $C$ be a connected component of $\mathcal{M}_{h}(s)\cap S$. As $S$ is proper ($S$ is an entire graph), $C\cap\alpha_{\phi}$ consists of a finite number of points. We will first show that $C$ meets a curve $\alpha_{\phi} \in \mathcal{F}$ in exactly one point.

Suppose $C\cap\alpha_{\phi_0}$ has more than one point. Since $\alpha_{\phi}$ is properly embedded, it has an injective parametrization $\alpha(t)$, $t \in (-\infty,+\infty)$. Let $t_1,t_2,$ such that $\alpha(t_1)$ and $\alpha(t_2)$ are consecutive points in $C\cap\alpha_{\phi}$. Denote $\tilde{C}$ the compact part of $C$ joining the points $\alpha(t_1)$ and $\alpha(t_2)$. Considering the foliation $\mathcal{F}$ of $M_{h}(s)$, since $\tilde{C}$ is a compact, there is a $\tilde\phi$ such that $\alpha_{\tilde\phi}$ is tangent to $\tilde{C}$ at a single point. But this is a contradiction with the fact that $\alpha_{\tilde\phi}$ has tangent line close to the vertical, and $\tilde{C}$ has tangent line away from the vertical (this arc is inside the graph).

Now suppose $C$ does not intersects some $\alpha_{\phi_{0}}$. Thus $C$ is contained in one of the components of $\mathcal{M}_{h}(s)\backslash \alpha_{\phi_{0}}$. As $C$ is proper and is inside the slab it has to intersect some $\alpha_{\bar\phi}$ in more than one point, but this contradicts the last paragraph. So $C$ meets each of the curves in $\mathcal{F}$ and does this in exactly one point. Therefore $C$ separates $S$ in two connected components.

Finally, we will show that $S\cap M_{h}(s)$ has exactly one connected component. Suppose this is not true, and let $C_{1}$ and $C_2$ two distinct connected components of this intersection. A curve $\alpha_{\phi}$ meet each $C_i$ transversely in exactly one point. The graph $S$ separates $\mathbb{E}(-1,\tau)$ in two connected components $U_1$ and $U_2$, such that $U_1$ is below $S$ and $U_2$ is above $S$. So, we can parametrize $\alpha_{\phi}$ as $\alpha(t)$, $t \in (-\infty,+\infty)$, in such a way that if $\alpha(r_i) \in \alpha_{\phi}\cap C_{i}$, then $\alpha(t) \in U_2$ for $t > r_i$, and $\alpha(t) \in U_1$ for $t < r_i$. Without loss of generality, suppose $r_2 > r_1$. It follows from the description that we can find $q_1 = \alpha(\bar{t})$, $q_{2} = \alpha(t^{*})$ in $U_2$ such that $\bar{t} \in (r_2,+\infty)$ and $t^{*} \in (r_{1},r_{2})$. Thus there is a path $\sigma$ joining $q_1,q_{2}$ that does not intersect $S$. The path $\sigma\cup\alpha|_{[\bar{t},t^{*}]}$ defines a closed curve which intersects $S$ in only one point. Since $S$ separates the ambient space, a closed curve which crosses $S$, intersects $S$ in a even number of points, this gives a contradiction.

\end{proof}

For each $s > 0$ consider the isometry of $\mathbb{H}^{2}$ given by $g_{s}(z) = s^{2}/\bar{z} - s/2$. This isometry sends the geodesic $\gamma_{s} = \{(s,t); t > 0\}$ to the geodesic joining the points $z = -s/2$ and $z = s/2$ at the asymptotic boundary of $\mathbb{H}^{2}$. We have $g_{s}(z) = -\overline{f_{s}}$, where $f_{s}$ is a positive isometry. Let $L_{s}$ be the isometry of $\mathbb{E}(-1,\tau)$ generated by $f_{s}$.

\begin{lem}[\cite{F.P}]
The argument of $f_{s}$ is $$\arg(f'_{s})(x + iy) = \arctan\biggl(\frac{-2xy}{x^2 - y^2}\biggr).$$ The image of $\gamma_{s,\theta}\times\{0\}$ under $L_{s}$ is an arc joining the points $(-s/2,0,-4\tau\theta_0)$ and $(s/2,0,0)$ whose projection on $\mathbb{H}^{2}$ is the equidistant $\beta_{s,\theta}$ to the geodesic joining the points $z = -s/2$ and $z = s/2$ of the asymptotic boundary, which makes an angle $\theta$ with the positive semi-axis at the point $z = s/2$.
\end{lem}
\begin{prop}[\cite{F.P}]
The image of $\mathcal{M}_{h}(s)$ under $L_{s}$ is a minimal surface which projects on the domain bounded by the curve $\beta_{s,\theta}\cup\{(x,0); -s/2 \leq x \leq s/2\}$, and whose asymptotic boundary is the union of
$$\{(x,0,t); t = -h + 2\tau\arcsin (1/d) \ \textrm{or} \ t = h - 2\tau\arcsin (1/d), -s/2 \leq x \leq s/2\},$$
and the vertical segments joining the end points of these arcs.
\end{prop}
\begin{proof}
Follows by direct calculations.
\end{proof}

\noindent
{\bf Remark 3}: Let $S$ be an entire vertical graph in $\mathbb{E}(-1,\tau)$ (half-space model) such that $S \subset \mathbb{H}^{2}\times [-h_{0},h_{0}]$, for some $h_{0} > 0$, and $\langle N,\partial_{t}\rangle \geq c_{0} > 0$, for some $c_{0}$, where $N$ is the unit normal of $S$. The isometry $L_{s}$ changes the height function by a bounded function. So applying proposition \ref{separation} we have that for $h > h_{0}$ large, $L_{s}\bigl(\mathcal{M}_{h}(s)\bigr)$ separates $S$ in two connected components.\label{separate}  

\begin{lem}\label{foliation}
Consider the isometries of $\mathbb{E}(-1,\tau)$ (in the half-space model) given by $F_{\lambda}(x,y,t) = (\lambda x, \lambda y, t)$, where $\lambda > 0$.
Then, for $h(d) > h_0$ large the family of surfaces $F_{\lambda}\circ L_{s}\bigl(\mathcal{M}_{h}(s)\bigr)$ foliates $\mathbb{H}^{2}\times [-h_{0},h_{0}]$.
\end{lem}
\begin{proof}
Clearly any point of $\mathbb{H}^{2}\times [-h_{0},h_{0}]$ is in some of the surfaces. Moreover given $\lambda_{1} \neq \lambda_{2}$, the map defined in coordinates by $(x,y,t) \mapsto \bigl(\frac{\lambda_2}{\lambda_1}x, \frac{\lambda_2}{\lambda_1}y, t\bigr)$ is an isometry of $\mathbb{E}(-1,\tau)$ (in the half-space model) which sends the surface $F_{\lambda_1}\circ L_{s}\bigl(\mathcal{M}_{h}(s)\bigr)$ into the surface $F_{\lambda_2}\circ L_{s}\bigl(\mathcal{M}_{h}(s)\bigr)$.
\end{proof} 

\begin{lem}\label{sub}
Consider the geodesic $\gamma_{s} = \{(s,t); t > 0\}$ of $\mathbb{H}^{2}$. Consider two entire graphs inside a slab of height $h_{0}$ that are disjoint and bounded away from the vertical. Let $\mathcal{R}$ be the region between them. Then, for $h > h_{0}$ sufficiently large, there exists $k_{0} = k_{0}(h,h_{0}) > 0$ such that $$\dist_{\mathbb{E}(-1,\tau)}\bigl(p,\Pi^{-1}(\gamma_{s})\bigr) \leq k_{0}, \forall p \in \mathcal{M}_{h}(s)\cap\mathcal{R}.$$
\end{lem}

\begin{proof}
Let $q \in \mathcal{M}_{h}(s)\cap\mathcal{R}$. The portion of $\mathcal{M}_{h}(s)$ inside $\mathcal{R}$ projects on a subdomain $\Omega'$ of $\Omega$ which is limited by two equidistants of $\gamma$. So the distance in $\mathbb{H}^{2}$ between $\gamma_{s}$ and $\Omega'$ is limited by a constant $c =  c(h_0)$, hence $\dist_{\mathbb{H}^{2}}(\gamma_{s},q) \leq c$. Consider a minimizing geodesic $\alpha$ of $\mathbb{H}^{2}$ joining $\gamma_{s}$ and $\Pi(q)$ and let $\tilde{\alpha}$ be an horizontal lift of $\alpha$ through a point $z \in \bigl(\Pi^{-1}(\gamma_{s})\bigr)\cap\mathcal{R}$. Denote by $\tilde{q}$ the point where $\tilde{\alpha}$ intersects the fiber over $q$. By lemma \ref{boundt} and the fact $z$ is inside a slab it follows that the $t$-coordinate of $\tilde{\alpha}$ is bounded. Hence $\dist_{\mathbb{E}(-1,\tau)}(\tilde{q},z)$ is bounded by a constant $\tilde{c}$. Then, 
\begin{eqnarray*}
\dist_{\mathbb{E}(-1,\tau)}(\Pi^{-1}(\gamma_{s}),q) &\leq & \dist_{\mathbb{E}(-1,\tau)}(z,q) \leq \dist_{\mathbb{E}(-1,\tau)}(z,\tilde{q}) + \dist_{\mathbb{E}(-1,\tau)}(\tilde{q},q)\\
&=& L(\tilde{\alpha}) + \dist_{\mathbb{E}(-1,\tau)}(\tilde{q},q) \leq c + \tilde{c}.
\end{eqnarray*}
\end{proof}

We finish this section describing the image of the surfaces $\mathcal{M}_{h}(s)$ under the isometry between the two models. Consider the maps $\phi$ and $\Phi$, which are isometries between the two models of $\mathbb{H}^{2}$ and $\mathbb{E}(-1,\tau)$ respectively. The geodesic $\phi(\gamma_{s})$ joins the points $\bigl(\frac{2s}{s^2 + 1},\frac{s^2 - 1}{s^2 + 1}\bigr)$ and $(0,1)$ at the asymptotic boundary. This curve divides the disc in two components. The equidistant $\gamma_{s,\arcsin(1/d)}$ is sent by $\phi$ to a equidistant $\beta_{s,d}$ of $\phi(\gamma_{s})$. Let $\mu_{s}$ be the arc of the asymptotic boundary joining the end points of $\phi(\gamma_{s})$ which is in the same component of $\beta_{s,d}$ and let $\Omega_{d}(s)$ be the domain limited by $\mu_{s}$ and $\beta_{s,d}$.

\begin{figure}[!ht]
\centering
\includegraphics[scale=0.7]{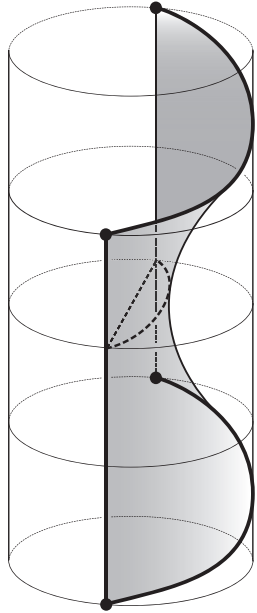}
\caption{The surface $\Phi\bigl(\mathcal{M}_{h}(s)\bigr)$}
\label{fig7}
\end{figure}

\begin{prop}
The image of $\mathcal{M}_{h}(s)$ under $\Phi$ is a minimal surface which projects on $\Omega_{d}(s)$. The asymptotic boundary is the union of 
\begin{eqnarray*}
\Biggl\{\Biggl(u,v,t + 4\tau\arctan\biggl(\frac{2u}{u^2 + (v - 1)^2 }\biggr)\Biggr); (u,v) \in \mu, \\
t = h(d) - 2\tau\arcsin (1/d) \ \textrm{or} \ t = h(d) + 2\tau\arcsin (1/d)\Biggr\}
\end{eqnarray*}
and the vertical segments joining the end points of these arcs.
\end{prop}
\begin{proof}
Follows by direct calculations.
\end{proof}

\section{The Slab Theorem}

\begin{defin}\label{slab}
\noindent
Consider the half-space model or the cylinder model of $\mathbb{E}(-1,\tau)$ (see section 2 and remark \ref{rem} in the introduction). We say that a region $\mathcal{R}$ of $\mathbb{E}(-1,\tau)$ is a generalized slab if the following conditions are satisfied: 
\begin{enumerate}
\item $\mathcal{R}$ is a domain bounded by two disjoint entire vertical graphs $S_1$ and $S_2$ with bounded height (in relation to $\mathbb{H}^{2}\times\{0\}$) and such that the tangent planes of $S_1$ and $S_2$ are bounded away from the vertical, i.e, $\langle N_{i},\xi\rangle \geq c > 0$, where $N_{i}$ is the unit normal of $S_{i}$ and $c$ is a constant;
\item There is a $C^{1}$ map $\Psi:\mathcal{R}\times\bigl(\mathbb{S}^{1}\times [-1,1]\bigr) \to \mathbb{E}(-1,\tau)$, such that, for each $p \in \mathcal{R}$ we have that $\mathcal{C}(p) := \Psi\bigl(p,\mathbb{S}^{1}\times [-1,1]\bigr)$ is a minimal annulus containing the point $p$ and its two boundary curves are disjoint from $\mathcal{R}$, one above $\mathcal{R}$ and one below $\mathcal{R}$. Moreover, any two annuli of the family $\{\mathcal{C}(p); p \in \mathcal{R}\}$ are isometric to each other.
\end{enumerate}
\end{defin}

The second condition on the previous definition is motivated by the following result due to Collin, Hauswirth and Rosenberg, which plays an essential role on the proof of the main theorem of this paper:
\begin{mainlem}\cite{C.H.R1,C.H.R2}\label{drag.lem}
Let $\Sigma$ be a properly immersed minimal surface in a complete Riemannian $3$-manifold $M$. Let $S$ be a compact surface with boundary and $f: S\times[0,1] \rightarrow M$ a $C^{1}$ map such that for each $0 \leq t \leq 1$, $S(t) := f(S\times\{t\})$ is a minimal surface of $M$. Suppose that:
\begin{description}
\item[a)] $\partial\bigl(S(t)\bigr)\cap \Sigma = \emptyset$ for $0 \leq t \leq 1$;
\item[b)] $S(0)\cap \Sigma \neq \emptyset$.
\end{description}
Then there is a continuous path $\gamma: [0,1] \rightarrow \Sigma$, $C^{1}$ almost everywhere, such that $\gamma(t) \in S(t)\cap \Sigma$, for $0 \leq t \leq 1$. Moreover one can prescribe any initial value to $\gamma(0) \in S(0)\cap \Sigma$.
\end{mainlem}

We just mention that the author in joint work with Ivan Passoni, proved a version of the {\it Dragging Lemma} concerning surfaces with constant mean curvature different from zero, \cite{L.P}.

We can now stablish our main result. The proof follows the ideas of \cite{C.H.R2}.
\begin{mainteo}
Let $\Sigma$ be a properly immersed minimal surface of finite topology in $\mathbb{E}(-1,\tau)$, $\tau \in \mathbb{R}$, inside of a generalized slab $\mathcal{R}$. Then each of its ends is a multi-graph. Moreover:
\begin{itemize}
\item[a)] If $\Sigma$ is embedded, then each of its ends contains an annulus that is a graph over the complement of a disc in $\mathbb{H}^{2}$;
\item[b)] If $\Sigma$ is embedded and has only one end, then it is an entire graph.
\end{itemize}
\end{mainteo}

\begin{proof}
Consider the half-space model. Suppose $\Sigma$ is a properly immersed minimal surface inside $\mathcal{R}$, which is homeomorphic to $\mathbb{S}^{1}\times\mathbb{R}^{+}$.
Choose a metric ball $B$ of $\mathbb{E}(-1,\tau)$, with center some $p_{0}$ in the interior of $\mathcal{R}$, and radius sufficiently large such that there are points of $\Sigma$ in $B$, $\partial\Sigma \subset B$ (in the case $\partial\Sigma \neq \emptyset$) and $\mathcal{C}(p_0) \subset B$, where $\mathcal{C}(p_0) := \Psi\bigl(p_{0},\mathbb{S}^{1}\times [-1,1]\bigr)$ is the annulus in the definition of a slab region. Since $\Sigma$ is properly immersed, the set $B\cap\Sigma$ has a finite number of compact connected components and there exists a compact set $\mathcal{K} \subset \mathbb{E}(-1,\tau)$, $B \subset \mathcal{K}$, such that any two points of $B\cap\Sigma$ can be joined by a path of $\Sigma$ in $\mathcal{K}$, see figure \ref{fig8}.

\begin{figure}[!ht]
\includegraphics[scale=0.6]{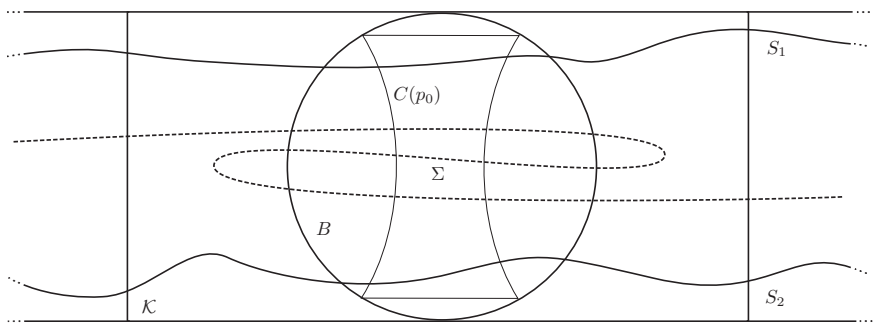}
\caption{The sets $B$ and $\mathcal{K}$}
\label{fig8}
\end{figure}

We will prove that if $p \in \Sigma$ is far enough from $\mathcal{K}$ the tangent plane $T_{p}\Sigma$ can not be vertical, hence $\Sigma$ is a multi-graph far from $\mathcal{K}$. This involves two steps:
\item[\bf Step 1]: If $T_{p}\Sigma$ is vertical, for $p \in \Sigma\backslash\mathcal{K}$, we can find a minimal surface $M$ such that $T_{p}M = T_{p}\Sigma$ and $M\cap\mathcal{K} \neq \emptyset$;
\item[\bf Step 2]: $\exists c > 0$ such that if $\dist(p,\mathcal{K}) > c$, then $M\cap\mathcal{K} = \emptyset$.
\\\\
{\bf Proof of step 1}:
\\

Suppose that $p \in \Sigma\backslash\mathcal{K}$ has a vertical tangent plane. The region $\mathcal{R}$ is inside a slab $\{-h_{0} < t < h_{0}\}$. Fix $h > h_{0}$ large enough such that the lemma \ref{transverse} is true. Applying an ambient isometry if necessary we can find a surface $M = L_{s}\bigl(\mathcal{M}_{h}(s)\bigr)$ which is tangent to $\Sigma$ at $p$ (see section 2). We want to prove that $M$ intersects $\mathcal{K}$. Assume the contrary. By remark \ref{separate}, the condition $\langle N_{i},S_{i}\rangle > c > 0$, of the boundary graphs of $\mathcal{R}$ guarantees that $M$ separates $\mathcal{R}$ in two connected components $M(-)$ and $M(+)$. We can assume $\mathcal{K} \subset M(+)$. There exists a neighbourhood $\mathcal{U}$ of $p$ in $\Sigma$ such that $\mathcal{U}\cap M$ consists of $2n$ arcs through $p$, $n > 1$, meeting at equal angles $\pi/n$ at $p$. Let $\sigma_{1}$ and $\sigma_{2}$ be distinct connected components of $\mathcal{U}\backslash M$, contained in $M(+)$.
\\\\
\textbf{Claim 1:}
$\sigma_{1}$ and $\sigma_{2}$ are contained in distinct components of $\Sigma\cap M(+)$. 

\begin{proof}
Suppose $\sigma_{1}$ and $\sigma_{2}$ are contained in the same connected component of $\Sigma\cap M(+)$. Then there is a path $\bar{\gamma}$ in $\Sigma\cap M(+)$, joining a point $q_{1} \in \sigma_{1}$ to a point $q_{2} \in \sigma_{2}$. We can also join $q_{1}$ to $q_{2}$ by a path $\tilde{\gamma}$ in $\mathcal{U}$ through $p$, with $\tilde{\gamma}\backslash\{p\} \subset M(+)$. Let $\Gamma := \bar{\gamma}\cup\tilde{\gamma} \subset \overline{M(+)}$. There are two possibilities: $\Gamma$ bounds a compact disc in $\Sigma$, or $\Gamma\cup\partial{\Sigma}$ bounds an immersed compact annulus. In both cases we denote by $D$ the surface that $\Gamma$ (or $\Gamma\cup\partial\Sigma$) bounds in $\Sigma$.

\begin{figure}[!ht]
\centering
\includegraphics[scale=0.7]{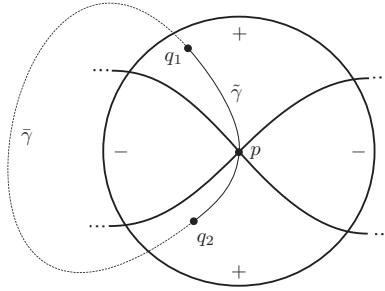}
\caption{The loop $\Gamma$}
\label{fig9}
\end{figure}

Consider the isometries $F_{\lambda}(x,y,t) = \bigl(\lambda x,\lambda y,t\bigr)$, where $\lambda > 0$. By lemma \ref{foliation}, the surfaces $M_{\lambda} := F_{\lambda}(M)$ foliates the slab $\{-h_{0} < t < h_{0}\}$. By the construction of $\Gamma$, $D$ contains points in $M(-)$ and $\partial D \subset M(+)$. Since $D$ is compact there is a $\lambda_{0}$ such that $M_{\lambda_{0}}$ and $D$ intersect and are tangent at a single point, but this contradicts the maximum principle since these two surfaces are minimal. Thus $\sigma_{1}$ and $\sigma_{2}$ are contained in distinct connected components $\Sigma_{1}$ and $\Sigma_{2}$ of $\Sigma\cap M(+)$. 
\end{proof}
\vspace{0.5cm}

Choose $\lambda > 0$ small, so that $M_{\lambda}$ intersects both $\sigma_{1}$ and $\sigma_{2}$, at points $x_{1} \in \Sigma_{1}$ and $x_{2} \in \Sigma_{2}$.
\\\\
\textbf{Claim 2:} The components $S_{x_{i}} := \Sigma_{i}\cap M_{\lambda}$ that contains $x_{i}$, $i = 1,2$, are non compact.

\begin{proof}
Suppose first that both components are compact. If $S_{x_{i}}$ bounds a disc, using the same idea of the previous claim we obtain a contradiction with the maximum principle. Suppose $S_{x_{1}}\cup S_{x_{2}}$ bounds a compact annulus $A$ in $\Sigma$ disjoint from $\partial{\Sigma}$, since $\partial A \subset M_{s}$, there is a $M_{t}$ tangent to $A$ at a single point and we obtain a contradiction. Hence both can not be compact.

Finally, suppose $S_{x_{1}}$ is not compact and $S_{x_{2}}$ is compact. Since $S_{x_{2}}$ is not null homotopic in $\Sigma$, $S_{x_{2}}\cup\partial\Sigma$ bound an immersed annulus $A$ and $\partial A \subset M_{s}(+)$, then again by the maximum principle $A \subset  M_{s}(+)$. The projection  is a Riemannian submersion, so
$$\dist_{\mathbb{E}(-1,\tau)}(p,q) \geq \dist_{\mathbb{H}^{2}}\bigl(\Pi (p),\Pi (q)\bigr), \forall p,q \in \mathbb{E}(-1,\tau).$$
Thus the distance between $M$ and $M_{\lambda}$ diverges at infinity, since this is true for their projections. Moreover $\Sigma$ is proper, so there are points $z \in S_{x_{1}}$ arbitrarily far from $M$. Choose such a $z$ so that the annulus $\mathcal{C}(z)$ through $z$ (see definition \ref{slab}) is contained in $M(+)$, and let $\eta:[0,l] \rightarrow \mathcal{R}\cap M(+)$ be a curve joining $z$ to $p_{0}$. Then, by the properties of $\mathcal{R}$ we can construct a family $f(t)$ of compact minimal annuli along $\eta$ satisfying, $f(0) = \mathcal{C}(z)$, $f(l) \subset B$ and $\partial\bigl(f(t)\bigr)\cap\mathcal{R} = \emptyset$, with one boundary circle above $\mathcal{R}$ and another below $\mathcal{R}$. Applying the Dragging Lemma to the family $f$ we obtain a path in $\Sigma\cap M(+)$ joining $z$ to a point $q \in \Sigma\cap B$. Join $q$ to a point of $\partial\Sigma$ by a path in $\Sigma\cap \mathcal{K}$. Let $\beta$ be the union of these paths. Since by hypothesis $\mathcal{K} \subset M(+)$, we have $\beta \subset \Sigma\cap M(+)$. But this contradicts the fact $S_{x_{1}}$ and $S_{x_{2}}$ are in distinct components of $\Sigma\cap M(+)$. This concludes the proof.
\end{proof}

\begin{figure}[!ht]
\centering
\includegraphics[scale=0.7]{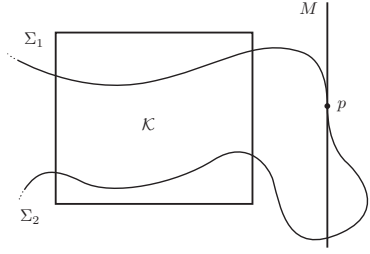}
\label{fig10}
\caption{Two noncompact components of $\Sigma$ in $M(+)$}
\end{figure}

Continuing the proof that M must intersect $\mathcal{K}$, we have two non compact components $S_{x_{i}} := \Sigma_{i}\cap M_{s}$, $i = 1,2$, with $\Sigma_{1}$ and $\Sigma_{2}$ in distinct components of $\Sigma\cap M(+)$. As above there exists points $z_{i} \in S_{x_{i}}$ far enough away from $M$ so that the annuli $\mathcal{C}(z_1)$ and $\mathcal{C}(z_2)$, through $z_1$ and $z_2$ respectively, are inside $M(+)$. Let $\eta_{i}$ be a curve in $\mathcal{R}$ joining $z_{i}$ to the point $p_{0}$ and consider families $f_{i}(t)$ of compact minimal annuli as before. Apply the Dragging Lemma to $f_{i}(t)$, to find a path in $\Sigma_{i}$ from $z_{i}$ to a point $q_{i} \in \Sigma\cap B$, $i = 1, 2$, and join $q_{1}$ to $q_{2}$ by a path in $\Sigma\cap \mathcal{K}$. This contradicts the fact that $z_{1}$ and $z_{2}$ are in distinct components of $\Sigma\cap M(+)$. Thus $M\cap \mathcal{K}$ is not empty.\\\\

\noindent
{\bf Proof of step 2}:
\\

Let $p \in \Sigma\backslash(\Sigma\cap\mathcal{K})$ be a point with vertical tangent plane. As before, applying an ambient isometry if necessary we can find a surface $\mathcal{M}_{h}(s)$ tangent to $\Sigma$ at $p$, with $h$ large as in proposition \ref{separation}. Consider the geodesic of $\mathbb{H}^{2}$ defined by $\gamma_{s}(r) = (s,r)$ which is associated to the surface $\mathcal{M}_{h}(s)$, we can also choose $s$ and $h$ such that $\Pi^{-1}(\gamma_{s})$ does not intersect the compact set $\mathcal{K}$.

Let $k_{0}$ be the constant of lemma \ref{sub}. Observe that
\begin{eqnarray*}
\dist_{\mathbb{E}(-1,\tau)}\bigl(\Pi^{-1}(\gamma_s), \mathcal{K}\bigr) &\geq & \dist_{\mathbb{E}(-1,\tau)}(\mathcal{K}, p) - \dist_{\mathbb{E}(-1,\tau)}\bigl(p,\Pi^{-1}(\gamma_s)\cap\mathcal{R}\bigr)\\ 
&\geq & \dist_{\mathbb{E}(-1,\tau)}(\mathcal{K}, p) - k_{0}.
\end{eqnarray*}

Suppose $\dist_{\mathbb{E}(-1,\tau)}(\mathcal{K}, p) \geq 3k_{0}$. It follows from the previous inequalities that $\dist_{\mathbb{E}(-1,\tau)}\bigl(\Pi^{-1}(\gamma_s), \mathcal{K}\bigr) \geq 2k_{0}$. Choose a point $q \in \mathcal{M}_{d}\cap\mathcal{K}$. Thus 
\begin{eqnarray*}
\dist_{\mathbb{E}(-1,\tau)}(\Pi^{-1}(\gamma_{s}),\mathcal{K}) \leq \dist_{\mathbb{E}(-1,\tau)}(\Pi^{-1}(\gamma_{s}),q) \leq k_{0},
\end{eqnarray*}
which is a contradiction. 

Thus, a point $p \in \Sigma$ satisfying $\dist_{\mathbb{E}(-1,\tau)}(p,\mathcal{K}) \geq 3k_{0}$ can not have a vertical tangent plane. 
\\\\
\textbf{Proof of the embedded case}:
\\

Suppose that $\Sigma$ as before is embedded. For $r > 0$, let $B_{r}(p_{0}) = \{p \in \mathbb{H}^{2}; \dist_{\mathbb{H}^{2}}(p,p_{0}) < r\}$, and define $\cyl(r) = \Pi^{-1}\bigl(\partial B_{r}(p_{0})\bigr)$. Let $r_{0} >0$ be large so that $\mathcal{E}$ is a multi-graph for $r \geq r_{0}$. Since $\Sigma$ is proper and its height is bounded, if $\Sigma \cap \cyl(r) \neq \emptyset$ for some $r \geq r_{0}$ then it is a finite union of embedded Jordan curves, each a graph over $\partial B_{r}(p_{0})$. These curves must be disjoint. Suppose the contrary and let $q$ be a point of intersection of two of these curves, then $\Sigma$ is tangent to $\cyl(r)$ at $q$, so $T_{q}\mathcal{E}$ is vertical, a contradiction.

Fix $r \geq r_{0}$ such that $\Sigma \cap \cyl(r) \neq \emptyset$ and let $\mathcal{E}_{r}$ be one connected component of the part of $\Sigma$ outside $\cyl(r)$. We will prove that $\mathcal{E}_{r}$ is a vertical graph over the complement of $B_{r}(p_{0})$. Let $N$ be the unit normal vector field of $\Sigma$ and define $\nu = \bigl\langle N,\partial_{t}\bigr\rangle$. Since $\mathcal{E}_{r}$ is a multi-graph, $\nu$ never vanishes on $\mathcal{E}_{r}$. Moreover $\nu$ is a Jacobi function on $\mathcal{E}_{r}$, thus $\mathcal{E}_{r}$ is stable and hence has bounded curvature. 

Then there is $\delta > 0$ such that, for any $p \in \mathcal{E}_{r}$ a neighbourhood of $p$ in $\mathcal{E}_{r}$ is a graph (in exponential coordinates) over the disc $D_{\delta} \subset T_{p}\Sigma$ of radius $\delta$, centered at the origin of $T_{p}\Sigma$. This graph, denoted by $G_{p}$, has bounded geometry. The number $\delta$ is independent of $p$ and the bound on the geometry of $G_{p}$ is uniform as well.

Consider the map $\Pi_{r}: \mathcal{E}_{r} \rightarrow \mathbb{H}^{2}\backslash B_{r}(p_{0})$, $\Pi_{r}(q,s) = q$. The absolute value of the Jacobian of $\Pi_{r}$ is $|\nu|$, so this map is a local diffeomorphism.
\\\\
\textbf{Claim 3:} The projection $\Pi_{r}: \mathcal{E}_{r} \rightarrow \mathbb{H}^{2}\backslash B_{r}(p_{0})$ is surjective. 

\begin{proof}Suppose $\Pi_{r}$ is not surjective. Then there exists a bounded open set $\Omega \subset \Pi_{r}(\mathcal{E}_{r})$ and $q_0 \in \partial \Omega$ such that, for some point $p \in \Pi_{r}^{-1}(\Omega)$, a neighbourhood of $p$ in $\mathcal{E}_{r}$ is a vertical graph of a function $f$ defined over $\Omega$ and this graph does not extend to a minimal graph over any neighbourhood of $q_0$.

We can suppose $\mathcal{R}$ contains $\mathbb{H}^{2}\times\{0\}$ and identify $\Omega$ with $\Omega\times\{0\}$. Let $\{q_n\} \subset \Omega$ be a sequence converging to $q_0$ and $p_{n} = \bigl(q_{n},f(q_{n})\bigr)$. Let $\Sigma_{n}'$ denote the image of $G_{p_{n}}$ under the vertical translation taking $p_{n}$ to $q_{n}$. There is a subsequence of $\{q_n\}$ (which we also denote by $\{q_n\}$) such that the tangent planes $T_{q_n}(\Sigma_{n}')$ converge to some vertical plane $P \subset T_{q_0} \bigl(\mathbb{E}(-1,\tau)\bigr)$. In fact, if this were not true, for $q_{n}$ close enough to $q_0$, the graph of bounded geometry $G_{p_n}$ would extend to a vertical graph beyond $q_0$. Hence $f$ would extend beyond $q_0$, a contradiction.

Thus, there exists a positive number $\delta_0 < \delta$ such that for $n$ big enough, a part of $\Sigma_{n}'$, denoted by $S_n$, is a graph over the disc of $P$ centered at $q_0$ with radius $\delta_0$ and this part contains the geodesic disc of $S_n$ centered at $p_n$ with radius $\delta/2$.

Using the Schauder theory, we can prove that a subsequence of $S_n$ converges in the $C^k$-topology, for any $k \in \mathbb{N}$, to a minimal surface $S$ passing through $q_0$ with vertical tangent plane $T_{q}S = P$. Let $N'$ be the limit unit normal field of $S$, by the maximum principle $\bigl\langle N', \partial_{t}\bigr\rangle$ is the null function on $S$. This means that $S$ is part of a vertical minimal cylinder over some curve $\gamma \subset \Omega$. Since $S$ is a minimal surface, $\gamma$ is a geodesic. Moreover, $S$ contains the geodesic disc of $\Pi_{r}^{-1}(\gamma)$ centered at $q_0$ with radius $\delta/2$.

Call $\tilde p \in \Pi_{r}^{-1}(p)$ the point in the same fiber as $p$ such that, $\tilde p$ lies above $p$, and the vertical distance between $\tilde p$ and $p$ is $\delta/4$. Observe that, by construction, $\tilde p$ is the limit of some sequence $\{\tilde{p_{n}}\} \subset D_{\delta}(q_n) \subset \Sigma_{n}'$. Therefore, the same arguments used above show that the geodesic disc of  $\Pi_{r}^{-1}(\gamma)$ centered at $\tilde p$ with radius $\delta/2$ is the limit of a sequence $\{\tilde{\Sigma}_{n}\}$, $\tilde{\Sigma}_{n} \subset \Sigma_{n}'$, extending in this way the part $\Sigma '$ of $\Pi_{r}^{-1}(\gamma)$. Repeating this argument, we can show that a connected part of $\Pi_{r}^{-1}(\gamma)$, which is arbitrarily high, is contained in the limit set of the sequence $\{\Sigma_{n}'\}$. So there are points of $\Sigma$ arbitrarily high. This gives a contradiction since $\Sigma$ has bounded height.
\end{proof}

This last claim proves there is $r_{1} > 0$ such that for each $r \geq r_{1}$, $\Sigma \cap \cyl(r)$ is not empty. Since $\Sigma$ is proper and is inside a slab, the map $\Pi_{r}$ is proper, as it is also a local diffeomorphism it follows that it is a smooth covering. Hence the number of components of $\Sigma \cap \cyl(r)$ is the same for each $r \geq r_{1}$. As $\mathcal{E}_{r}$ has only one component this number has to be 1. Thus $\mathcal{E}_{r}$ is an graph over the complement of $B_{r}(p_{0})$.\\

Now, suppose $\Sigma$ is a properly embedded minimal surface with finite topology and one end. Let $r \geq r_1$. Denote $\beta_{r} = \Sigma \cap \cyl(r)$ and let $\Sigma_{r}$ be the compact part of $\Sigma$ with boundary $\beta_{r}$. 
\\\\
\textbf{Claim 4:} $\Sigma_{r}$ is a graph over the closure of the ball $B_{r}(p_{0})$.

\begin{proof}
Exactly like the proof of proposition 5.2 in \cite{Y}.

\end{proof}

The claims 3 and 4 prove that $\Sigma$ is an entire graph.

\end{proof}

\section{Examples}

\subsection{Example 1}

Consider the disc model of $\mathbb{H}^{2}$. For each $z_{0} \in \mathbb{H}^{2}$, the map defined by $f_{z_0}(z) = \displaystyle\frac{z - z_{0}}{\bar{z_{0}}z - 1}$ is an positive isometry that sends $(0,0)$ to $z_{0}$. Consider an isometry of $\mathbb{E}(-1,\tau)$ of the form $$F_{z_0}(z,t) = \bigl(f_{z_0}(z), t - 2\tau\arg f'_{z_0}(z) + 2\tau\pi\bigr).$$

Denote $z_0 = x_0 + iy_0$, $z = x + iy$, then $$f'_{z_0}(z) = \frac{|z_0|^2 - 1}{(\bar z_{0}z - 1)^2} = \frac{(|z_0|^2 - 1)(z_{0}\bar{z} - 1)^2}{|\bar z_{0}z - 1|^2},$$
so we can define the following angle function
$$\arg f'_{z_0}(x + iy) = \pi + \arctan\Biggl[\frac{2(x_{0}x + y_{0}y - 1)(xy_{0} - x_{0}y)}{(x_{0}x + y_{0}y - 1)^2 - (xy_{0} - x_{0}y)^{2}}\Biggr],$$
hence $$|-2\tau\arg f'_{z_0}(z) + 2\tau\pi| < |\tau|\pi.$$

Note that $F_{z_0}(0,0) = (z_0,0)$, $\forall z_{0} \in \mathbb{H}^{2}$. Given $0 < \epsilon <\pi\sqrt{1 + 4\tau^2} - 2|\tau|\pi$, choose a rotational catenoid $\mathcal{C}_{d_{\epsilon}}$ (described on subsection \ref{cat}) with height satisfying $\mathcal{H}_{d_{\epsilon}} > \displaystyle\frac{\pi}{2}\sqrt{1 + 4\tau^2} - \epsilon$, and consider the family $\mathcal{F} = \bigl\{F_{z_0}\bigl(\mathcal{C}_{d_{\epsilon}}\bigr); z_0 \in \mathbb{H}^{2}\bigr\}$. We have $$\frac{\pi}{2}\sqrt{1 + 4\tau^2} - |\tau|\pi - \frac{\epsilon}{2} < \mathcal{H}_{d_{\epsilon}} - |\tau|\pi < \mathcal{H}_{d_{\epsilon}} - 2\tau\arg f'_{z_0}(z) + 2\tau\pi.$$

Therefore the slab of height $\pi\sqrt{1 + 4\tau^2} - |\tau|\pi - \epsilon$ is a region satisfying the conditions of our theorem.

\subsection{Example 2}

Fix $\tau \in \mathbb{R}$. The preliminary calculations made here does not depend on the particular model, however at the end of the subsection we specify the model to obtain  specific examples. Let $\Sigma \subset \mathbb{E}(-1,\tau)$ be an entire graph of some function $u$, such that $\langle N_{\Sigma},\partial_{t}\rangle \geq c > 0$, for some constant $c$. We define the variation of $u$ in $\Omega \subset \mathbb{H}^{2}$ as $$V_{u}(\Omega) := \sup\{|u(q_1) - u(q_2)|; q_1, q_2 \in \Omega\}.$$ 

Let $B_{r}(p)$ the geodesic ball of $\mathbb{H}^{2}$, with center $p$ radius $r$. Let $\Gamma(p) = \bigl\{\bigl(q,u(q)\bigr); q \in \partial B_{r}(p)\bigr\}$. Given $h > 0$, consider the vertical translates of $\Gamma$ by $h$ and $-h$ and denote these curves by $\Gamma_{h}(p)$ and $\Gamma_{-h}(p)$ respectively. 

Since $\Gamma(p)$ is a graph over $\partial B_{r}(p)$ and this curve is convex, using the same arguments of claim 4 in the proof of the theorem, we can prove that the solution of the plateau problem for $\Gamma_{h}(p)$ (also for $\Gamma_{-h}(p)$) is unique and is a graph of a function $v^{+}$ over $B_{r}(p)$ (respectively $v^{-}$), denote this solution by $D_{h}$ (respectively $D_{-h}$). The Douglas criterion says that if there is an annulus with boundary $\Gamma_{h}(p)$ and $\Gamma_{-h}(p)$ whose area is less than $|D_{h}| + |D_{-h}|$ then, there exists a least area minimal annulus whose boundary circles are these two curves.

Let $C_{h}$ be the cylinder with boundary $\Gamma_{h}(p)$ and $\Gamma_{-h}(p)$, i.e, the surface parametrized by $\Psi(s,t) = \bigl(x(s), y(s), z(s) + t \bigr)$, where $\bigl(x(s), y(s), z(s)\bigr)$ is a parametrization of $\Gamma$ and $-h \leq t \leq h$. Let $a$ and $b$ defined as in subsection \ref{cmcgraph}. We have that
\begin{eqnarray*}
|D_{\pm h}| &=& \int_{B_{r}(p)} \sqrt{1 + a^2 + b^2} \ d\mu \geq |B_{r}(p)| = 2\pi(\cosh r - 1);\\\\
|C_{h}|&=& 2|\Pi(\partial B_{r}(p))|h = 4\pi h\sinh r.
\end{eqnarray*}

Thus if 
\begin{equation}\label{height}
h < \frac{\cosh r - 1}{\sinh r},
\end{equation}
we have $|C_{h}| < |D_{h}| + |D_{-h}|$.

Choose a function such that $V_{u}(B_{r}(p)) < \displaystyle\frac{\cosh r - 1}{\sinh r}, \forall p \in \mathbb{H}^{2}$, and such that $\Sigma$ has bounded height. Consider $h > V_{u}(B_{r}(p))$ satisfying inequality \eqref{height}.

Now, consider the half-space model. Applying a vertical translation if necessary, we can suppose $u(i) = 0$. For each $x_{0} + y_{0}i \in \mathbb{H}^{2}$, the map defined by $G_{x_{0},y_{0}}(z,t) = \bigl(y_{0}z + x_{0}, t + u(x_{0},y_{0})\bigr)$ is an isometry that sends $(i,0)$ to $\bigl(x_{0} + y_{0}i,u(x_{0},y_{0})\bigr)$. Let $A_{h}$ a solution to the Douglas problem with boundary $\Gamma_{h}(i)$ and $\Gamma_{-h}(i)$. Then $\bigl\{G_{x_{0},y_{0}}(A_{h}); x_{0} + y_{0}i \in \mathbb{H}^{2}\bigr\}$ is a continuous family of compact minimal annuli. Moreover, by the last paragraph all the annuli in this family have boundary circles disjoint from $\Sigma$, one above $\Sigma$ and the other one below. Take two vertical translates of $\phi(\Sigma)$ such that this last condition is still satisfied for them. Now apply the isometry $\Phi$ between the two models, then we obtain a generalized slab in $\mathbb{E}(-1,\tau)$.
\\

Now we specialize in the case of $\mathbb{H}^{2}\times\mathbb{R}$. Consider an function satisfying 
\begin{equation}\label{grad}
|\nabla u(q)| \leq C <  \displaystyle\frac{\cosh r - 1}{2r\sinh r}, \ \forall q \in \mathbb{H}^{2}.
\end{equation}
So, the variation of $u$ in $B_{r}(p)$ is bounded by $2Cr$ and there is $h$ satisfying
$$2Cr < h < \displaystyle\frac{\cosh r - 1}{\sinh r}.$$

For example, the function $u: \{(x, y) \in \mathbb{R}^{2}; y > 0\} \rightarrow \mathbb{R}$, defined by $u(x,y) = \int_{0}^{y}\frac{\sin t}{t}dt$, is bounded and satisfies \eqref{grad}.

Now, in the cylinder model of $\mathbb{H}^{2}\times\mathbb{R}$, consider the function $v(x,y) = \alpha x + \beta y$. We have $$|\nabla v(x,y)| = (1 - x^{2} - y^{2})\sqrt{\alpha^{2} + \beta^{2}} \leq \sqrt{\alpha^{2} + \beta^{2}}.$$
For obtain the family of annuli we apply the isometries defined by $F_{z_{0}}(z,t) = \biggl(\displaystyle\frac{z - z_{0}}{\bar{z_{0}}z - 1},t + v(z_{0})\biggr), \ z_{0} \in \mathbb{H}^{2}$, to a solution of the Douglas problem with boundary $\Gamma_{h}(0)$ and $\Gamma_{-h}(0)$. The generalized slab is obtained taking appropriated vertical translates of the graph of $v$. If $\alpha = \frac{\pi}{2} + \epsilon$ (where $\epsilon > 0$) and $\beta = 0$, the region between the two translates is not contained in any horizontal slab of height $\pi$, so this region is not in the conditions of the slab theorem of Collin, Hauswirth and Rosenberg.

\noindent \textsc{Instituto de Matem\'atica e Estat\'istica, UERJ\\ Rua S\~ao Francisco Xavier, 524\\Pavilh\~ao Reitor Jo\~ao Lyra Filho, 6º andar - Bloco B\\ 20550-900, Rio de Janeiro-RJ, Brazil}
\\\\
\noindent {\it Email address}: \texttt{vanderson@ime.uerj.br}

\end{document}